\documentclass[leqno]{amsart}
\usepackage{amsmath}
\usepackage{amssymb}
\usepackage{amsthm}
\usepackage{enumerate}
\usepackage[mathscr]{eucal}
\usepackage{xcolor}
\theoremstyle{plain}
\usepackage{tikz}
\usepackage{soul}
\usepackage[normalem]{ulem}
\newtheorem{theorem}{Theorem}[section]
\newtheorem{lemma}[theorem]{Lemma}
\newtheorem{prop}[theorem]{Proposition}
\theoremstyle{definition}
\newtheorem{definition}[theorem]{Definition}
\newtheorem{remark}[theorem]{Remark}

\newtheorem{example}[theorem]{Example}

\newtheorem{cor}[theorem]{Corollary}
\theoremstyle{remark}

\begin{document}
\title[On best  coapproximations in subspaces of diagonal matrices]{On best coapproximations in subspaces of diagonal matrices}
\author[Sain, Sohel, Ghosh and  Paul  ]{Debmalya Sain, Shamim Sohel, Souvik Ghosh and Kallol Paul }

\newcommand{\acr}{\newline\indent}
\address[Sain]{Department of Mathematics\\ Indian Institute of Science\\ Bengaluru 560012\\ Karnataka \\INDIA}
\email{saindebmalya@gmail.com}

\address[Sohel]{Department of Mathematics\\ Jadavpur University\\ Kolkata 700032\\ West Bengal\\ INDIA}
\email{shamimsohel11@gmail.com}

\address[Ghosh]{Department of Mathematics\\ Jadavpur University\\ Kolkata 700032\\ West Bengal\\ INDIA}
\email{sghosh0019@gmail.com}

\address[Paul]{Department of Mathematics\\ Jadavpur University\\ Kolkata 700032\\ West Bengal\\ INDIA}
\email{kalloldada@gmail.com}

\thanks{The research of Dr. Debmalya Sain is sponsored by SERB N-PDF fellowship under the mentorship of Professor Apoorva Khare. Dr. Sain feels grateful to have the opportunity to acknowledge the loving friendship of his childhood friend Dr. Subhamoy Neogi, an accomplished physician with an empathetic heart. The second and third author would like to thank  CSIR, Govt. of India, for the financial support in the form of Junior Research Fellowship under the mentorship of Prof. Kallol Paul. The research of Prof. Paul  is supported by project MATRICS (MTR/2017/000059)  of  SERB, DST, Govt. of India.}

\subjclass[2010]{Primary 46B20, Secondary 47L05}
\keywords{best coapproximations; coproximinal subspace; co-Chebyshev subspace; diagonal matrices; Birkhoff-James orthogonality }

\begin{abstract}
We characterize the best coapproximation(s) to a given matrix $ T $ out of a given subspace $ \mathbb{Y} $ of the space of diagonal matrices $ \mathcal{D}_n $, by using Birkhoff-James orthogonality techniques and with the help of a newly introduced property, christened the $ * $-Property. We also characterize the coproximinal subspaces and the co-Chebyshev subspaces of $ \mathcal{D}_n $ in terms of the $ * $-Property. We observe that a complete characterization of the best coapproximation problem in $ \ell_{\infty}^n $ follows directly as a particular case of our approach.
\end{abstract}

\maketitle
\section{Introduction.} The study of best coapproximations in Banach spaces was initiated by Franchetti and Furi in \cite{FF}, as a complementary notion to the classical theory of best approximations. Unlike the Hilbert space case, the existence and the uniqueness of best coapproximations are known to be  difficult problems in the setting of Banach spaces, especially from a computational point of view. We refer the readers to \cite{FF, PS} for more information on this topic. The purpose of the present article is to study the best coapproximation problem in subspaces of diagonal matrices, from the perspective of orthogonality. Indeed, we obtain a complete characterization of the best coapproximations to any given square matrix, in the said subspaces of matrices. It should be noted that recently in \cite{S}, the concept of orthogonality has been utilized to address best approximation problems in Banach spaces. Although our present exploration is motivated in spirit by the said approach, the techniques used in this study are completely different from those used in \cite{S}. Let us now introduce the relevant notations and the terminologies to be used throughout the article.\\

We use the symbol $ \mathbb{H} $ to denote a Hilbert space, along with its usual inner product $ \langle~, \rangle $ and its usual norm $ ||.||_2.$ In this article, we will only work with \emph{real} Hilbert spaces. Let $ \perp $ denote the usual orthogonality relation on $ \mathbb{H}. $ $ \mathbb{L}(\mathbb{H}) $ ($ \mathbb{K}(\mathbb{H}) $) denotes the Banach space of all bounded (compact) linear operators on $ \mathbb{H}, $ endowed with the usual operator norm. Given $T \in \mathbb{L}(\mathbb{H}),$ let $M_T$ denote the norm attainment set of $T$,i.e, $ M_T=\left\lbrace x \in H ~ :~ \|x\|_2=1  , \|Tx\|_2=\|T\|\right\rbrace .$  We note that $M_T \neq \phi$ whenever $T \in \mathbb{K}(\mathbb{H}).$ In case $ \mathbb{H} $ is finite-dimensional, given any $ T \in \mathbb{L}(\mathbb{H}), $ we identify $ T $ with its matrix representation with respect to the canonical basis of $ \mathbb{H}. $ Let $ \mathcal{M}_n $ denote the space of all $ n \times n$ real  matrices and let $ \mathcal{D}_n $ be the subspace of $ \mathcal{M}_n, $ consisting of diagonal matrices. Given $T \in \mathcal{M}_n$, let $T^t$ denotes the transpose of $T.$  Given any $ A \in \mathcal{D}_n $ with diagonal entries $ a_{ii},~ 1 \leq i \leq n, $ we write $ A = ((a_{11}, a_{22}, \ldots, a_{nn})), $ for the sake of brevity. The zero element of $\mathbb{R}^n$ is denoted by $\theta, $ whenever $ n > 1. $ The following definition is of fundamental importance in our entire study:

\begin{definition}
	Let $ (\mathbb{X}, \|.\|) $ be a Banach space and let $ \mathbb{Y} $ be a subspace of $ \mathbb{X}. $ Given any $ x \in \mathbb{X}, $ we say that $ y_0 \in \mathbb{Y} $ is a best coapproximation to $ x $ out of $ \mathbb{{Y}} $ if $ \| y_0 - y \| \leq \| x - y \| $ for all $ y \in \mathbb{Y}. $
\end{definition}

In general, neither the existence nor the uniqueness of best coapproximation(s) is guaranteed, even in the finite-dimensional case. A subspace $ \mathbb{Y} $ of the Banach space $ \mathbb{X} $ is said to be coproximinal if a best coapproximation to any element of $ \mathbb{X} $ out of $ \mathbb{Y} $ exists. A coproximinal subspace $ \mathbb{Y} $ is said to be co-Chebyshev if the best coapproximation is unique in each case. Given $x \in \mathbb{X}$ and a subspace $\mathbb{{Y}}$ of $ \mathbb{X}, $ the (possibly empty) set of all best coapproximations to $x$ out of $\mathbb{{Y}}$ is denoted by $\mathcal{R}_\mathbb{{Y}}(x).$ Our aim in this article is to explore the problem of finding the best coapproximation(s) to any given $ T \in \mathcal{M}_n $ out of any given subspace $ \mathbb{Y} $ of $ \mathcal{D}_n, $ provided the best coapproximation(s) exist. We employ Birkhoff-James orthogonality techniques and the concept of numerical range of an operator $ T \in \mathbb{L}(\mathbb{H}), $ to obtain a complete solution to the above problem, which is also computationally effective. Let us recall from the pioneering articles \cite{B, J} that given any two elements $ x, y $ in a Banach space $ \mathbb{X}, $ we say that $ x $ is Birkhoff-James orthogonal to $ y, $ written as $ x \perp_B y, $ if $ \| x+\lambda y \| \geq \| x \| $ for all $ \lambda \in \mathbb{R}. $ It should be noted that given any subspace $ \mathbb{Y} $ of a Banach space $ \mathbb{X} $ and an element $ x \in \mathbb{X}, $ $ y_0 \in \mathbb{Y} $ is a best coapproximation to $ x $ out of $ \mathbb{Y} $ if and only if $ \mathbb{Y} \perp_B (x-y_0,) $ i.e., $ y \perp_B (x-y_0) $ for all $ y \in \mathbb{Y}. $ Using Theorem $ 1.1 $ of \cite{BS}, also known as the Bhatia-$\breve{S}$emrl Theorem, we study the best coapproximation problem from the perspective of Birkhoff-James orthogonality. We also recall that given any $ T \in \mathbb{L}(\mathbb{H}), $ the numerical range of $ T $ is defined as $ W(T) := \{ \langle Tx, x \rangle : \| x \|_2 = 1 \}. $ We refer the readers to \cite{GR}, for a comprehensive study and possible applications of the numerical range of an operator in $ \mathbb{L}(\mathbb{H}). \\$

In order to apply the above concepts in our designated study, we need to introduce the following definitions whose importance will be self-evident in due course of time.

\begin{definition} \label{component}
	Let $ \mathcal{A}= \{A_1, A_2, \ldots, A_m\}  $ be a set of linearly independent elements in $ \mathcal{D}_n, $  where  $A_k= ((a^k_{11}, a^k_{22}, \ldots, a^k_{nn})),  $ for each $ 1 \leq k \leq m. $  Considering the diagonal matrices  $ A_1, A_2, \ldots, A_m $ as column vectors,   we form the $ n \times m $ matrix $ \widetilde{A}= (\widetilde{a_{ij}})_{ 1 \leq i \leq n, 1 \leq j \leq m},$ where $ \widetilde{a_{ij}}= a_{ii}^j.$ 
	\begin{enumerate}
		\item [(i)]  For each $ i \in \{1,2,\ldots,n\}, $ the $i$-th component  of  $\mathcal{A}$ is defined  as the $i$-th row of $\widetilde{A},$ i.e., 
		$\left( a_{ii} ^1 ,a_{ii}^2 ,\ldots,a_{ii} ^m \right). $ Whenever the context is clear we simply say the $i$-th component of $\mathcal{A}$ as the $i$-th component.
		\item  [(ii)]  The $i$-th component and the $j$-th component are said to  be equivalent if 
		$\left( a_{ii} ^1 ,a_{ii}^2 ,\ldots,a_{ii} ^m \right) $ = $\pm \left( a_{jj} ^1 ,a_{jj}^2 ,\ldots,a_{jj} ^m \right)  $.
		\item  [(iii)]  The positively associated set $\ P_i^+(\mathcal{A})$ of the $i$-th component is defined  as 
		\[ P_i^+(\mathcal{A}) = \Big\{ j \in \{1,2, \ldots, n\}: \left( a_{jj} ^1 ,a_{jj}^2 ,\ldots,a_{jj} ^m \right)  =  \left( a_{ii} ^1 ,a_{ii}^2 ,\ldots,a_{ii} ^m \right) \Big\}.\]
		Similarly, the negatively associated set $\ P_i^- (\mathcal{A})$ is defined as 
		\[ P_i^- (\mathcal{A})= \Big\{ j \in \{1,2, \ldots, n\}: \left( a_{jj} ^1 ,a_{jj}^2 ,\ldots,a_{jj} ^m \right)  =  -\left( a_{ii} ^1 ,a_{ii}^2 ,\ldots,a_{ii} ^m \right) \Big\}.\]
		For simplicity, we write $ P_i^+(\mathcal{A})= P_i^+$ and $ P_i^-(\mathcal{A})= P_i^-,  $ if the context is clear.
		\item  [(iv)]  The $i$-th component is said to satisfy the $ * $-Property with respect to $\mathcal{A}$ if there exist 	$\beta_1, \beta_2, \ldots, \beta_m \in \mathbb{R} $ such that 
		\[  \Big| \sum_{k=1}^{m} \beta_ka_{ii}^k \Big| > \max     \Big\{ \Big| {\sum_{k=1}^{m} \beta_k a_{jj}^k} \Big|:1 \leq  j \leq n, ~  j \notin   P_i^+ \cup P_i^- \Big\} .\]
	\end{enumerate}

\end{definition}

\begin{definition}
	Let $ \mathcal{A}= \{ A_1 , A_2 ,\ldots, A_m \} $ be linearly independent in $  \mathcal{D}_n ,$  where  $A_k= ((a^k_{11}, a^k_{22}, \ldots, a^k_{nn}))  ,$ for each  $ 1 \leq k \leq m. $ 
	Suppose that the $i$-th component satisfies the $*$-Property   with $ |P_i^+ \cup P_i^- | = k_i.$
	 Given $T= (b_{pq} )  _{1 \leq p, q \leq n}$, we define the $*$-associated matrix of $T$ corresponding to the $i$-th component as a square matrix of order $ k_i,$ given by  $ ^iT_*=(c_{rs})_{1 \leq r,s \leq k_i},$ where 
	 \begin{eqnarray*}
	 c_{rs} & = & b_{rs}, ~  (r, s) \in P_i^+ \times \left(P_i^+ \cup P_i^-\right) \\
	                & = & - b_{rs},  (r, s) \in P_i^- \times \left(P_i^+ \cup P_i^-\right) 
	 \end{eqnarray*} 
\end{definition}

 In this paper, we obtain a complete characterization of the best coapproximation to an element of $ \mathcal{M}_n $ out of a  given subspace of $ \mathcal{D}_n.$   We  emphasize that our method is computationally convenient and it is possible to present a tractable algorithmic solution to the above problem by using it. We further illustrate this by presenting explicit numerical examples in support of our claim. The first step in this direction is to obtain a theoretical characterization of the best coapproximation problem in $ \mathbb{K}(\mathbb{H}). $ The second step is to explore some fundamental attributes of the newly introduced $ * $-Property in connection with the best coapproximation problem. In the final step, we assimilate the previously obtained results to present the desired algorithm to study the best coapproximation problem in any given subspace of $ \mathcal{D}_n. $ We also characterize the coproximinal subspaces and co-Chebyshev subspaces of $ \mathcal{D}_n $ in $ \mathcal{M}_n. $ As another important application of the present study, we observe that a particular case of our method gives a complete solution to the best coapproximation problem in $ \ell_{\infty}^m, $ for any given $ m \in \mathbb{N}. $

\section{Main Results}   
We begin with a theoretical characterization of best coapproximations in $ \mathbb{K}(\mathbb{H}), $ that will play a crucial role in the computational approach towards finding best coapproximation$\left(s\right)$ (provided it exists) in any given subspace of $ \mathcal{D}_{n}, $ as adopted in the present article.

\begin{theorem}\label{general}
	Let $ \mathbb{H} $ be a Hilbert space and let $ T, A_1 , A_2 , \ldots, A_m \in \mathbb{K}(\mathbb{H}). $ Given any $ \alpha_1,\alpha_2, \ldots,\alpha_m \in \mathbb{R} , $  $ \sum_{i=1}^{m} \alpha_i A_i $ is a best coapproximation to $ T $ out of $span \{  A_1 , A_2 ,$ $ \ldots, A_m \} $ if and only if given any $\beta_1,\beta_2, \ldots,\beta_m \in \mathbb{R}$, there exists $x \in M_{\sum_{i=1}^{m} \beta_iA_i}$ such that $\langle ~ \sum_{i=1}^{m} \beta_iA_ix, (T-\sum_{i=1}^{m} \alpha_iA_i)x ~\rangle =0. $
\end{theorem}
\begin{proof}
	It follows from the definitions of Birkhoff-James orthogonality and best coapproximation that $ \sum_{i=1}^{m} \alpha_iA_i $ is a best coapproximation to $ T $ out of $ span \{ A_1 , A_2 ,$ $ \ldots, A_m \} $ if and only if $ A \perp_B \left(T - \sum_{i=1}^{m} \alpha_iA_i\right), $ for all $ A \in span \left\lbrace  A_1, A_2, \ldots, A_m \right\rbrace. $ Clearly, this is equivalent to the following:
	
	\[ \sum_{i=1}^{m} \beta_iA_i \perp_B (T-\sum_{i=1}^{m} \alpha_iA_i) ~\forall ~\beta_1,\beta_2, \ldots,\beta_m \in \mathbb{R}. \]
	
	It follows from   Theorem $ 2.2 $ of  \cite{SP} that for any $\beta_1,\beta_2, \ldots,\beta_m \in \mathbb{R},$ $ M_{\sum_{i=1}^{m} \beta_iA_i} $ is the unit sphere of some subspace of $\mathbb{H}. $ Now applying Theorem $ 2.2 $ of \cite{SMP}, we conclude that the above condition is equivalent to the existence of $ x = x(\beta_1, \ldots, \beta_m) \in M_{\sum_{i=1}^{m} \beta_iA_i}  $ such that $ \langle ~ \sum_{i=1}^{m} \beta_iA_ix, (T-\sum_{i=1}^{m} \alpha_iA_i)x ~\rangle =0, $ for any $ \beta_1,\beta_2, \ldots,\beta_m \in \mathbb{R}. $ This completes the proof of the theorem.
	
	\end{proof}

We  establish some fundamental attributes of the newly introduced $*$-Property which also plays an important role in our scheme. To begin with, we establish the basis   invariance of equivalent components and the  $*$-Property.

\begin{prop}
	 Let $Y$ be a subspace of $ \mathcal{D}_n$ and let $\mathcal{A} = \{ A_1, A_2,\ldots, A_m \}, \mathcal{B}= \{ B_1, B_2,\ldots,$ $ B_m \}$ be two bases of $ Y,$
	where $A_k= ((a^k_{11}, a^k_{22}, \ldots, a^k_{nn}))  $and $ B_k = ((b^k_{11}, b^k_{22}, \ldots, b^k_{nn})) $, for each $ 1 \leq k \leq m .$  Then \\
	$ (i) $ $~P_i^+(\mathcal{A})= P_i^+(\mathcal{B}) ~\mbox{and}~ P_i^-(\mathcal{A})=P_i^-(\mathcal{B}), ~i\in \{1,2,\ldots,n\}.$ \\
	$(ii) $ For any $i \in \{1,2,\ldots,n \},$ the $i$-th component satisfies the $*$-Property  with respect to $\mathcal{A}$  if and only if the $i$-th component satisfies the $*$-Property with respect to  $\mathcal{B}.$
\end{prop}

\begin{proof}  (i) Consider the two matrices  $\widetilde{A}$ and $\widetilde{B}$ as constructed  in Definition \ref{component}. Since $\mathcal{A}$ and $ \mathcal{B}$ are two bases of $Y,$ so  there exists an invertible matrix $ Q= (q_{ij})_{1 \leq i, j \leq m}$ such that $ \widetilde{A}= \widetilde{B}Q,$ where $ \widetilde{a_{ij}}= \sum_{k=1}^{m} \widetilde{b_{ik}} q_{kj},$ for any $ 1 \leq i \leq n, 1 \leq j \leq m.$ The desired result then  follows easily.\\
  (ii) We first prove the necessary part.   As before let  $Q=(q_{ij})_{1 \leq i, j \leq m}$ be the invertible matrix such that $\widetilde{A}= \widetilde{B}Q,$ where $  \widetilde{a_{ij}}= \sum_{k=1}^{m} \widetilde{b_{ik}} q_{kj},$ for any $ 1 \leq i \leq n, 1 \leq j \leq m.$ Since the $i$-th component satisfies the $*$-Property with respect to $\mathcal{A},$ there exist $\beta_1,\beta_2, \ldots,\beta_m \in \mathbb{R} $ such that
 	  \[  | \sum_{k=1}^{m} \beta_ka_{ii}^k | > \max     \left\lbrace | {\sum_{k=1}^{m} \beta_k a_{jj}^k} |:1 \leq  j \leq n, ~  j \notin P_i^+ \cup P_i^- \right\rbrace. \]
 	 Observe that  $\widetilde{A} \widetilde{\beta}= \widetilde{B} Q\widetilde{\beta}, $ where $\widetilde{\beta}= ( \beta_1 ~ \beta_2 ~\ldots ~\beta_m)^t.$ Considering  $\widetilde{\gamma}= ( \gamma_1 ~ \gamma_2~ \ldots ~ \gamma_m)^t $ $ = Q\widetilde{\beta},$ 
 	  it is easy to see that for any $r \in \{ 1, 2, \ldots, n\},$ 
 	 \begin{eqnarray*}
 	 |\sum_{k=1}^{m} \gamma_kb_{rr}^k | =  |  \sum_{k=1}^{m} \Big(\sum_{j=1}^{m} q_{kj}\beta_j\Big) \widetilde{b_{rk}}| = | \sum_{j=1}^{m} ( \sum_{k=1}^{m} \widetilde{b_{rk}}  q_{kj} ) \beta_j | = | \sum_{j=1}^{m} \widetilde{a_{rj}} \beta_j | = | \sum_{j=1}^{m} \beta_j a_{rr}^j|. 
 	\end{eqnarray*}
 	  This immediately shows that the $i$-th component satisfies the $*$-Property with respect to $\mathcal{B}.$ This completes the necessary part.
   The sufficient part follows similarly.
\end{proof}

In light of the above theorem, from now onwards we will not explicitly mention the choice of basis in the description of the $ * $-Property. Our next theorem essentially guarantees the existence of the $ * $-Property.

\begin{theorem}\label{existence}
Let $ \mathcal{A}= \{ A_1 , A_2 ,\ldots, A_m \} $ be linearly independent in $  \mathcal{D}_n ,$  where  $A_k= ((a^k_{11}, a^k_{22}, \ldots, a^k_{nn}))  ,$ for each  $ 1 \leq k \leq m. $    Then there exists $ 1 \leq i \leq n $ such that the $i$-th component satisfies the $*$-Property.
\end{theorem}

\begin{proof}
	Let the $i_1$-th, $i_2$-th, \ldots,  $i_p$-th components represent all the nonequivalent components For any $ \widetilde{w} = \left( \gamma_1, \gamma_2, \ldots, \gamma_m\right)  \in \mathbb{R}^m,$ consider the set of scalars
	\[ S_{\widetilde{w}}:= \left\lbrace  |\sum_{k=1}^{m}\gamma_ka_{i_1i_1}^k|,|\sum_{k=1}^{m}\gamma_ka_{i_2i_2}^k|, \ldots, |\sum_{k=1}^{m}\gamma_ka_{i_pi_p}^k| \right\rbrace .	\]
	
 	$ Case~ 1:$ If $S_{\widetilde{w}} $ attains its maximum at a unique point, say at $|\sum_{k=1}^{m}\gamma_ka_{i_ri_r}^k|,$ where $r \in \{1, 2, \ldots, p\},$  then clearly the $i_r$-th component satisfies the $*$-Property. \\

$Case~ 2:$ Let us assume that the maximum of $S_{\widetilde{w}}$ is attained at  exactly two points. Suppose that for  $i_s, i_t \in \{i_1, i_2, \ldots, i_p\} $ and $ i_s \neq i_t,$ 
		 \[ | \sum_{k=1}^{m} \gamma_ka_{i_si_s}^k |= | \sum_{k=1}^{m} \gamma_ka_{i_ti_t}^k |> max     \left\lbrace | {\sum_{k=1}^{m} \gamma_k a_{qq}^k} |: q \in \{i_1, i_2, \ldots, i_p\} \setminus \{i_s, i_t\}\right\rbrace .\]
		 
		 Let us define functions $ f_s,f_t : \mathbb{R}^m \longrightarrow \mathbb{R} $ given by
		  \[ f_s(\widetilde{u}) := | \langle~\widetilde{u}, \widetilde{a}_{i_s}~\rangle |~\textit{and}~ f_t(\widetilde{u}) := | \langle~\widetilde{u}, \widetilde{a}_{i_t}~\rangle |, \]	
		 where $ \widetilde{u}=\left(\alpha_1, \alpha_2,\ldots, \alpha_m \right) \in \mathbb{R}^m  $ and $\widetilde{a}_{i_s}= (a_{i_si_s}^1, a_{i_si_s}^2, \ldots, a_{i_si_s}^m), \widetilde{a}_{i_t} =  (a_{i_ti_t}^1, a_{i_ti_t}^2,$ $ \ldots, a_{i_ti_t}^m)$ are  the $i_s$-th and  the $i_t$-th component, respectively.\\
		 
		 Let us also define another function, $ g : \mathbb{R}^m \longrightarrow \mathbb{R} $ given by 
		 
		 \[ g(\widetilde{u}) := max \left\lbrace  | \langle ~\widetilde{u}, \widetilde{a}_q ~\rangle  |: q \in \{i_1, i_2, \ldots, i_p\} \setminus \{i_s, i_t\}\right\rbrace, \]
		 \\
		 where $\widetilde{a}_q = (a_{qq}^1, a_{qq}^2, \ldots, a_{qq}^m)$ is   the $q$-th component.  Since $f_i, g $ are continuous function on $\mathbb{R}^m,$  $ \phi_i = f_i - g $ is also continuous and $ \phi_i\left( \widetilde{w} \right) > 0 ,$ for all $i \in \left\lbrace s, t\right\rbrace ,$ where $ \widetilde{w} = \left( \gamma_1,\gamma_2, \ldots , \gamma_m \right) \in \mathbb{R}^m. $ It is easy to observe that there exists an open ball $ \mathcal{B}_\delta\left( \widetilde{w} \right),  $ with radius $ \delta > 0 $ and centered at $ \widetilde{w}, $ such that $ \phi_i \left( \widetilde{y} \right) > 0, $ for all $\widetilde{y} \in   \mathcal{B}_\delta\left( \widetilde{w}  \right).$ Consider the hyperspaces $ H_1, H_2 $ of $ \mathbb{R}^m $ given by 
		 	\[ H_1=\left\lbrace \widetilde{x} \in \mathbb{R}^m : \langle~ \widetilde{x} , \left( \widetilde{a}_{i_s} + \widetilde{a}_{i_t} \right)\rangle =0 \right\rbrace,  
		 	\] 
		 	\[ H_2 =\left\lbrace \widetilde{x} \in \mathbb{R}^m : \langle~ \widetilde{x} , \left( \widetilde{a}_{i_s} - \widetilde{a}_{i_t} \right) \rangle =0 \right\rbrace.   
		 	\]
		 		We note that $ \left\lbrace  \widetilde{x} \in \mathbb{R}^m : f_s \left( \widetilde{x}\right)  = f_t\left( \widetilde{x}\right) \right\rbrace \ = H_1 \cup H_2 ,$
		 	which is a nowhere dense set  in $\mathbb{R}^m .$ Therefore, by choosing $ \widetilde{v}:=(\beta_1, \beta_2, \ldots, \beta_m) \in \mathcal{B}_\delta\left( \widetilde{w}  \right) \setminus (H_1 \cup H_2), $ we obtain that $ f_s\left( \widetilde{v} \right) \neq f_t\left( \widetilde{v} \right). $  Without loss of generality, assume   $ f_s\left( \widetilde{v} \right) > f_t\left( \widetilde{v} \right). $  It is now easy to observe that
		 	\[   | \sum_{k=1}^{m} \beta_ka_{i_si_s}^k | > max     \left\lbrace | {\sum_{k=1}^{m} \beta_k a_{qq}^k} |:1 \leq  q \leq n, ~  q \notin   P_{i_s}^+ \cup P_{i_s}^- \right\rbrace.
		 	\] 
		 	Therefore, the $i_s$-th component satisfies the $*$-Property. \\
		 	
		 	$Case ~ 3:$ Suppose that the maximum of $S_{\widetilde{w}}$ is attained at $r(>2)$ number of points and let the $i_1$-th,$ i_2$-th, \ldots, $i_r$-th components  satisfy 
		 	\[ | \sum_{k=1}^{m} \gamma_ka_{i_1i_1}^k |=\ldots = | \sum_{k=1}^{m} \gamma_ka_{i_ri_r}^k |  > max     \left\lbrace | {\sum_{k=1}^{m} \gamma_k a_{qq}^k} |: q \in \{i_1,\ldots, i_p\} \setminus \{i_1,\ldots, i_r\} \right\rbrace. \]  By similar argument as given in Case $2$,   it can be shown that at least one of the $i_l$-th components satisfies the $*$-Property, where $l \in \left\lbrace 1, 2, \ldots, r\right\rbrace .$ This completes the theorem.
\end{proof}

\begin{remark}\label{remark}
	Let $ \mathcal{A}= \{ A_1 , A_2 ,\ldots, A_m \} $ be linearly independent in $  \mathcal{D}_n ,$  where  $A_k= ((a^k_{11}, a^k_{22}, \ldots, a^k_{nn}))  ,$ for each  $ 1 \leq k \leq m. $   In particular, for any $ \beta_1, \beta_2, \ldots, \beta_m \in \mathbb{R},$ there exists an $i$-th component such that $ || \sum_{k=1}^{m}\beta_kA_k || = | \sum_{k=1}^{m}\beta_ka_{ii}^k|,$ where the $i$-th component satisfies the $*$-Property.
\end{remark}

Our next aim is to obtain a tractable necessary and sufficient condition for the $*$-Property. In this context, we first recall the definition of a normal cone. A non-empty set $K \subset \mathbb{R}^n$ is said to be a normal cone if the following three conditions are satisfied:
\[ (i) u,v \in K \Rightarrow u+v \in K, (ii)  u \in K,~ \alpha \geq 0 \Rightarrow \alpha u \in K, (iii) K \cap (-K) = \{\theta\}.\]  
We define interior  of the normal cone $K$, denoted by $ int(K),$ as the collection of all interior points of the normal cone $K$. We refer the readers to \cite{SPM}, for an application of the notion of normal cones in studying approximate Birkhoff-James orthogonality in Banach spaces. We also require the following lemma for our purpose.

\begin{lemma}
Let $ \mathcal{A}= \{ A_1 , A_2 ,\ldots, A_m \} $ be linearly independent in $  \mathcal{D}_n ,$  where  $A_k= ((a^k_{11}, a^k_{22}, \ldots, a^k_{nn}))  ,$ for each  $ 1 \leq k \leq m. $    Given any $i, j \in \{1, 2, \ldots,n\},$ where $j \notin P_i^+ \cup P_i^-,$ there exist a pair of normal cones whose interiors are the collection of all the $\left( \beta_1, \beta_2, \ldots, \beta_m\right) \in \mathbb{R}^m $  such that $| \sum_{k=1}^{m}\beta_ka_{ii}^k| > | \sum_{k=1}^{m}\beta_ka_{jj}^k|. $ 
\end{lemma}

\begin{proof}
	Let us consider the set 
	\[ C := \left\lbrace \left( \beta_1, \beta_2, \ldots, \beta_m\right) \in \mathbb{R}^m : ~| \sum_{k=1}^{m}\beta_ka_{ii}^k| > | \sum_{k=1}^{m}\beta_ka_{jj}^k| \right\rbrace .
	\]

	Let us also construct two sets $K_1$ and $K_2$ such that 
	\[ K_1 := \left\lbrace \left( \beta_1, \beta_2, \ldots, \beta_m\right) \in C : ~ \sum_{k=1}^{m}\beta_ka_{ii}^k > 0 \right\rbrace \bigcup ~\{\theta \}, \]
  \[  K_2 := \left\lbrace \left( \beta_1, \beta_2, \ldots, \beta_m\right) \in C:~ \sum_{k=1}^{m}\beta_ka_{ii}^k < 0 \right\rbrace \bigcup ~\{\theta \}.
	\]
	From the definition of $K_1$ and $K_2,$ it is evident that $K_1 = -K_2.$ Now, it is immediate that $ \widetilde{x} \in K_1$ implies that $ \alpha\widetilde{x} \in K_1, $ for all $\alpha \geq 0.$ Therefore, to prove that $K_1$ is a normal cone, we only need to show $\widetilde{u},\widetilde{v} \in K_1$ implies $\widetilde{u} + \widetilde{v} \in K_1.$
	Suppose that $\widetilde{u}=(\alpha_1, \alpha_2,\ldots, \alpha_m)$ and $\widetilde{v}=(\gamma_1, \gamma_2, \ldots, \gamma_m) \in K_1.$ Then, $ \sum_{k=1}^{m}(\alpha_k + \gamma_k)a_{ii}^k >0$ and for any $ j \in \{1, 2, \ldots, n\} $ such that $j \notin P_i^+ \cup P_i^-,$ it follows that
	\begin{eqnarray*}
		 |\sum_{k=1}^{m}(\alpha_k + \gamma_k)a_{jj}^k  | \leq |\sum_{k=1}^{m}\alpha_ka_{jj}^k| + |\sum_{k=1}^{m} \gamma_ka_{jj}^k| & <& |\sum_{k=1}^{m}\alpha_ka_{ii}^k| + |\sum_{k=1}^{m} \gamma_ka_{ii}^k|\\ &= & |\sum_{k=1}^{m}(\alpha_k + \gamma_k)a_{ii}^k|.
	\end{eqnarray*}
This proves that $ K_1 $ (and therefore, $K_2$) is a normal cone. It is rather straightforward to verify that $int(K_1) = K_1 \setminus \{\theta\} $ and $int(K_2) = K_2 \setminus \{\theta\}. $ Therefore, $ int(K_1) \cup int(K_2) =C,$ as desired. This completes the proof of the lemma.
	\end{proof}

Next we introduce the notion of associated pair of cones, which turns out to be useful in characterizing the $*$-Property.

\begin{definition}
		Let $ \mathcal{A}= \{ A_1 , A_2 ,\ldots, A_m \} $ be linearly independent in $  \mathcal{D}_n ,$  where  $A_k= ((a^k_{11}, a^k_{22}, \ldots, a^k_{nn}))  ,$ for each  $ 1 \leq k \leq m. $  Given any $i \in \{1, 2, \ldots, n\},$ we define the  pair of normal cones $\  K_j^i,  -K_j^i\ $ as the associated pair of  cones of the $i$-th component with respect to the $j$-th component,  given by
		\[ K_j^i := \left\lbrace \left( \beta_1, \beta_2, \ldots, \beta_m\right) \in \mathbb{R}^m: ~ | \sum_{k=1}^{m}\beta_ka_{ii}^k| > | \sum_{k=1}^{m}\beta_ka_{jj}^k|~ \textit{and}~ \sum_{k=1}^{m}\beta_ka_{ii}^k > 0 \right\rbrace \cup \{\theta\}, \]
		 for all $j \notin P_i^+ \cup P_i^-.$ 
 \end{definition}
Finally, we are in a position to characterize the $*$-Property from a geometric perspective.
\begin{theorem}
		Let $ \mathcal{A}= \{ A_1 , A_2 ,\ldots, A_m \} $ be linearly independent in $  \mathcal{D}_n ,$  where  $A_k= ((a^k_{11}, a^k_{22}, \ldots, a^k_{nn}))  ,$ for each  $ 1 \leq k \leq m. $     Then for any $ i \in \{1, 2, \ldots, n\},$ the $i$-th component satisfies the $*$-Property if and only if~  \[ \bigcap\left\lbrace int(K_j^i ) \cup int(-K_j^i) : 1 \leq j \leq n, j \notin P_i^+ \cup P_i^-  \right\rbrace \neq \phi. \]
\end{theorem}

\begin{proof}
	Suppose that the $i$-th component satisfies the $*$-Property, i.e., there exists  $\widetilde{x}=(\beta_1, \beta_1, \ldots, \beta_m) \in \mathbb{R}^m $ such that \[|\sum_{k=1}^{m} \beta_ka_{ii}^k| > max\left\lbrace |\sum_{k=1}^{m} \beta_ka_{jj}^k | : j \notin  P_i^+ \cup P_i^-\right\rbrace. \] 
	This is equivalent to   $\widetilde{x} \in int(K_j^i) \cup int(-K_j^i),$ for all $1 \leq j \leq n $ and $j \notin  P_i^+ \cup P_i^-, $ where $ K_j^i, -K_j^i$ are the pair of associated cones of the $i$-th component with respect to the $j$-th component. Therefore,
	 \[ \bigcap\left\lbrace int(K_j^i )\cup int(-K_j^i) : 1 \leq j \leq n, j \notin P_i^+ \cup P_i^-  \right\rbrace \neq \phi. 
	 \]
	  This completes the proof of the necessary part of the theorem. We note that the sufficient part of the theorem also follows from similar arguments and the definition of pair of associated cones. This establishes the theorem.
	
\end{proof}

We next obtain a simple and useful sufficient condition for the $ * $-Property. It should be noted that in practise, the following result can be readily applied in most cases of the computations involving the $*$-Property, since checking the linear independence of a given set of vectors is not complicated at all by virtue of the well-known method of row reduction of matrices.

\begin{prop}\label{proposition:prop1} Let $ \mathcal{A}= \{ A_1 , A_2 ,\ldots, A_m \} $ be linearly independent in $  \mathcal{D}_n ,$  where  $A_k= ((a^k_{11}, a^k_{22}, \ldots, a^k_{nn}))  ,$ for each  $ 1 \leq k \leq m. $ 
	Suppose that the  $i$-th component   $\left( a_{ii} ^1 ,a_{ii}^2 ,\ldots,a_{ii} ^m \right)  \notin span\{\left( a_{jj} ^1 ,a_{jj}^2 ,\ldots,a_{jj} ^m \right) :1 \leq  j \leq n, ~  j \notin   P_i^+$ $ \cup P_i^-\}, $  
	where $\left( a_{jj} ^1 ,a_{jj}^2 ,\ldots,a_{jj} ^m \right) $  is  the $j$-th component.	Then the $ i $-th component satisfies the $ * $-Property.
\end{prop}
\begin{proof}
	Let $ Y_1= span \{\left( a_{jj} ^1 , a_{jj}^2 ,\ldots, a_{jj} ^m \right):1 \leq  j \leq n, ~  j \notin   P_i^+ \cup P_i^-\} $ and let $Y_2=  span \{\left( a_{jj} ^1 , a_{jj}^2 ,\ldots, a_{jj} ^m \right) :1 \leq  j \leq n \}.$ Clearly, $ Y_1 \subsetneq Y_2 = \mathbb{R}^m ,$ which implies that $ Y_2^\perp \subsetneq Y_1^\perp .$ Therefore, there exists $  \left( \gamma_1,\gamma_2,\ldots ,\gamma_m \right) \in Y_1^\perp \setminus Y_2^\perp$ such that  $ | \sum_{k=1}^{m} \gamma_ka_{ii}^k | > max     \left\lbrace | {\sum_{k=1}^{m} \gamma_k a_{jj}^k} |:1 \leq j \leq n, ~  j \notin   P_i^+ \cup P_i^- \right\rbrace = 0 $. In other words, the $i$-th component satisfies the $*$-Property, as desired.  
\end{proof}

\begin{remark}
	Suppose that $\mathcal{T}_i $ is the collection of all those $j$ such that the $j$-th component is  a scalar multiple of the $i$-th component, where $ i, j \in \{1, 2, \ldots, n\}$. Let us assume that the $i$-th component $ ( a_{ii} ^1 ,a_{ii}^2 ,\ldots,a_{ii} ^m ) =c_j  ( a_{jj} ^1 ,a_{jj}^2 ,\ldots, $ $ a_{jj} ^m  ),$ where $\left( a_{jj} ^1 ,a_{jj}^2 ,\ldots, a_{jj} ^m \right) $ is the $j$-th component and $ |c_j| \geq 1, $ for all $j \in \mathcal{T}_i .$ Also assume that  $( a_{ii} ^1 ,a_{ii}^2 ,\ldots,$ $ a_{ii} ^m )  \notin span\{ \left( a_{kk} ^1 ,a_{kk}^2 ,\ldots,a_{kk} ^m \right) :1 \leq  k \leq n, ~ $ $  k \notin \mathcal{T}_i \}  .$ Following similar argument from Proposition \ref{proposition:prop1}, the $i$-th component satisfies the $*$-Property. 
\end{remark}

We are now ready to present a computationally convenient characterization of the best coapproximation   to an  element of $ \mathcal{M}_n $ out of a given subspace of $ \mathcal{D}_n. $

\begin{theorem}\label{characterization}
	Let $ \mathcal{A}= \{ A_1 , A_2 ,\ldots, A_m \} $ be linearly independent in $  \mathcal{D}_n ,$  where  $A_k= ((a^k_{11}, a^k_{22}, \ldots, a^k_{nn}))  ,$ for each  $ 1 \leq k \leq m. $  
	Suppose that the $j_1$-th, $j_2$-th, $\ldots, j_r$-th nonequivalent components satisfy  the $*$-Property. Then given any $T \in \mathcal{M}_n,~  \sum_{k=1}^{m} \alpha_kA_k $ is a best coapproximation to $T$ out of $span \left\lbrace A_1, A_2,\ldots, A_m \right\rbrace $ if and only if  $\alpha_1, \alpha_2,\ldots, \alpha_m \in \mathbb{R}$ satisfy the following relations:
	\[ a_{{j_p}{j_p}}^1 \alpha_1 +  a_{{j_p}{j_p}}^2 \alpha_2 +\ldots + a_{{j_p}{j_p}}^m \alpha_m \in W \left( ^{j_p}T_*  \right), 
	\]
	for all $p \in \{1, 2,\ldots , r \} ,$ where $  W \left( ^{j_p}T_*  \right) $ is the numerical range of the   $*$-associated matrix of $T$ corresponding to the $j_p$-th component.
	\end{theorem}

\begin{proof}
	
	Let us first prove  the necessary part of the theorem. Assume that the $j_s$-th component satisfies the $*$-Property, where $s \in \left\lbrace 1, 2, \ldots, r \right\rbrace $. Then there exists $ \widetilde{v} = ( \beta_1, \beta_2, $ $ \ldots, \beta_m )  \in \mathbb{R}^m $ such that $ | \sum_{k=1}^{m} \beta_ka_{j_sj_s}^k | > max     \{ | {\sum_{k=1}^{m} \beta_k a_{qq}^k} |:1 \leq  q \leq n, ~  q \notin  $ $  P_{j_s}^+ \cup P_{j_s}^- \}  $.
	 Therefore, $ \| \sum_{k=1}^{m} \beta_kA_k\| = | \sum_{k=1}^{m} \beta_ka_{j_sj_s}^k | $ and the norm attainment set of $\sum_{k=1}^{m} \beta_k A_k $ is
	 \[ M_{\sum_{k=1}^{m} \beta_k A_k} = \left\lbrace x=(x_1, x_2,\ldots, x_n) \in \mathbb{R}^n~ :~ \| x \|_2= 1, x_h = 0 ~ \forall~ h \notin P_{j_s}^+ \cup P_{j_s}^-\right\rbrace.
	 \]
	Let $T =(b_{uv})_{1 \leq u,v \leq n} \in \mathcal{M}_n.$ Since $  \sum_{k=1}^{m} \alpha_kA_k $ is a best coapproximation to $T$ out of $span \{ A_1, A_2, $ $ \ldots, A_m \}, $ it follows from Theorem \ref{general} that $ \langle ~ \sum_{k=1}^{m} \beta_kA_kx,~ (T-\sum_{k=1}^{m} \alpha_kA_k)x ~\rangle = 0 ,$ i.e., $  \langle ~ \sum_{k=1}^{m} \beta_kA_kx,~ Tx \rangle = \langle ~\sum_{k=1}^{m} \beta_kA_kx,~\sum_{k=1}^{m} \alpha_kA_kx ~\rangle  ,$ for some $ x \in M_{\sum_{k=1}^{m} \beta_k A_k }. $  By a straight forward calculation, the previous equation can be expressed as 

	\begin{eqnarray*}
			a_{{j_s}{j_s}}^1 \alpha_1  +\ldots + a_{{j_s}{j_s}}^m \alpha_m &=& \sum_{u \in P_{j_s}^+,  v \in P_{j_s}^+\cup P_{j_s}^-}b_{uv}x_ux_v   	- \sum_{u \in P_{j_s}^-,  v \in P_{j_s}^+\cup P_{j_s}^-}b_{uv}x_ux_v \\ &\in & W \left( ^{j_s}T_*  \right).
	\end{eqnarray*}
Similarly, we can observe that for all $ p \in \left\lbrace 1, 2, \ldots, r \right\rbrace, $
	 \begin{equation}
	 \begin{aligned}
	  a_{{j_p}{j_p}}^1 \alpha_1 +  a_{{j_p}{j_p}}^2 \alpha_2 +\ldots + a_{{j_p}{j_p}}^m \alpha_m \in W \left( ^{j_p}T_*  \right),
	 \end{aligned}
	 \end{equation}
	 completing the proof of the necessary part.\\
	
	  We now prove the sufficient part of the theorem. For any $ \beta_1, \beta_2,\ldots, \beta_m \in \mathbb{R} ,$ not all zero, by virtue of Remark \ref{remark}, there exists   $t \in \left\lbrace 1, 2, \ldots, r\right\rbrace$ such that the $j_t$-th component satisfies $\| \sum_{k=1}^{m} \beta_kA_k\| = | \sum_{k=1}^{m} \beta_k a_{j_tj_t}^k | .$ Let $ P_{j_t}^+ = \left\lbrace j_t,j_{t_2},\ldots, j_{t_w} \right\rbrace $ and $ P_{j_t}^- = \left\lbrace j_{t_{w+1}}, j_{t_{w+2}},\ldots, j_{t_v}  \right\rbrace ,$ so that $| P_{j_t}^+ \cup P_{j_t}^- |= v. $ From the hypothesis, $ \alpha_1, \alpha_2, \ldots, \alpha_m \in \mathbb{R} $ satisfy the following relations:
	  	\[ a_{{j_t}{j_t}}^1 \alpha_1 +  a_{{j_t}{j_t}}^2 \alpha_2 +\ldots + a_{{j_t}{j_t}}^m \alpha_m \in W \left( ^{j_t}T_*  \right). 
	  \]
	    Therefore, there exists  $ y = \left( y_t, y_{t_2}, \ldots, y_{t_v} \right) \in \mathbb{R}^v $ with $ \|y\|_2 = 1 $ such that 
	  \begin{equation}\label{equation2}
	  \begin{aligned}
	   a_{{j_t}{j_t}}^1 \alpha_1 +  a_{{j_t}{j_t}}^2 \alpha_2 +\ldots + a_{{j_t}{j_t}}^m \alpha_m = \langle  ^{j_t}T_*y , ~y \rangle.
	  \end{aligned}
	  \end{equation}
	   Now by taking $ \widehat{y}=( \widetilde{y} _1, \widetilde{y}_2,\ldots,  \widetilde{y}_n) \in \mathbb{R}^n~$ such that $~  \widetilde{y}_h = 0 ~ \forall~ j_h \notin P_{j_t}^+ \cup P_{j_t}^- $ and $\widetilde{y}_h = y_h ~\forall~ j_h \in  P_{j_t}^+ \cup P_{j_t}^- $, it is easy to observe that $ \widehat{y} \in  M_{\sum_{k=1}^{m} \beta_k A_k}.$ By some easy calculations and by using the equation (\ref{equation2}), we conclude that 
	   \begin{equation*}
	   \begin{aligned}
        \langle \sum_{k=1}^{m} \beta_k A_k\widehat{y},~(T-\sum_{k=1}^{m} \alpha_kA_k)~\widehat{y} ~\rangle = (\sum_{k=1}^{m} \beta_ka_{j_tj_t}^k)\left[  \langle y,~ ^{j_t}T_*y \rangle - \sum_{k=1}^{m} \alpha_ka_{j_tj_t}^k\right]  = 0.
	   \end{aligned}
	   \end{equation*}
	   The sufficient part of the theorem now follows directly from Theorem \ref{general}. This establishes the theorem.
	   
	  \end{proof}

  \begin{remark}
    Suppose that $ A_1, A_2, \ldots, A_m \in \mathcal{M}_n$, where $ 1 \leq m \leq n , $ are such that $ A_iA_j^t,A_i^tA_j $ are symmetric, for all $i, j \in \left\lbrace  1, 2, \ldots, m \right\rbrace. $ Then	from Corollary $9$ of \cite{MM}, there exist orthogonal matrices $ P $ and $ Q$  such that $ P^tA_iQ = D_i $, where $ D_i \in \mathcal{D}_n$, for all $i \in \left\lbrace  1, 2, \ldots, m \right\rbrace $. Moreover, since $P$ and $ Q $ are orthogonal matrices, it is easy to see that $ \| \sum_{i=1}^{m} \beta_iA_i\| = \|\sum_{i=1}^{m} \beta_iD_i\|,  $ for all $  \beta_1, \beta_2,\ldots, \beta_m \in \mathbb{R}. $
    \end{remark}

The above remark allows us to present the following strengthened version of Theorem \ref{characterization}.

\begin{theorem}
	 Let $ \mathcal{A}= \{ A_1 , A_2 ,\ldots, A_m \} $ be linearly independent in $  \mathcal{M}_n $  such that  $ A_iA_j^t,A_i^tA_j $ are symmetric, for all $i, j \in \left\lbrace  1, 2, \ldots, m \right\rbrace .$  Let $ D_1, D_2, \ldots, D_m \in \mathcal{D}_n
	 $ be such that  $ D_i = P^tA_iQ, $ where $ D_i=(( d_{11}^i, d_{22}^i, \ldots, d_{nn}^i))$,  for all $i \in \left\lbrace  1, 2, \ldots, m \right\rbrace $ and $ P, Q \in \mathcal{M}_n $ are orthogonal matrices. Suppose that the $j_1$-th, $j_2$-th, $\ldots , j_r$-th nonequivalent components satisfy  the $*$-Property  (with respect to $span\{  D_1, D_2,\ldots, $ $ D_m\}  $). Then given any $T \in \mathcal{M}_n,~  \sum_{i=1}^{m} \alpha_iA_i $ is a best coapproximation to $T$ out of $span \left\lbrace A_1, A_2,\ldots, A_m \right\rbrace $ if and only if  $\alpha_1, \alpha_2,\ldots, \alpha_m \in \mathbb{R}$ satisfy the following relations:
	\[ d_{{j_p}{j_p}}^1 \alpha_1 +  d_{{j_p}{j_p}}^2 \alpha_2 +\ldots + d_{{j_p}{j_p}}^m \alpha_m \in W \left( ^{j_p} (P^tTQ)_*  \right), 
	\]
	for all $p \in \{1, 2,\ldots , r\} ,$ where $  W \left( ^{j_p} (P^tTQ)_*   \right) $ is the numerical range of the $*$-associated matrix of $P^tTQ$ corresponding to the $j_p$-th component. 
\end{theorem}

\begin{proof}
$ \sum_{i=1}^{m} \alpha_iA_i $ is a best coapproximation to $T$ out of $span \left\lbrace A_1, A_2,\ldots, A_m \right\rbrace $ if and only if given any $\beta_1, \beta_2, \ldots , \beta_m \in \mathbb{R} ,$ there exists $ x \in M_{\sum_{i=1}^{m}\beta_iA_i}$ such that \\ $ \langle~ \left(  \sum_{i=1}^{m} \beta_iA_i \right) x, \left( T-\sum_{i=1}^{m} \alpha_iA_i \right) x ~ \rangle = 0 ,$ i.e,
 
 \[ \langle  ~\left( P \sum_{i=1}^{m} \beta_iD_iQ^t \right) x, \left( T-P\sum_{i=1}^{m} \alpha_iD_iQ^t \right) x ~ \rangle = 0 . \] 
 
  So, for $y=Q^tx$, it is immediate that   $ \langle ~\left(  \sum_{i=1}^{m} \beta_iD_i \right) y, \left( P^tTQ-\sum_{i=1}^{m} \alpha_iD_i \right) y ~ \rangle = 0 .$ We also note that $ x \in M_{\sum_{i=1}^{m}\beta_iA_i} $ if and only if $ y= Q^tx \in M_{\sum_{i=1}^{m}\beta_iD_i}$. Therefore, $ \sum_{i=1}^{m} \alpha_iA_i $ is a best coapproximation to $T$ out of $span \left\lbrace A_1, A_2,\ldots, A_m \right\rbrace $ if and only if $ \sum_{i=1}^{m} \alpha_iD_i $ is a best coapproximation to $P^tTQ$ out of $span \left\lbrace D_1, D_2,\ldots, D_m \right\rbrace .$ Now the desired result follows directly from Theorem \ref{characterization}. This completes the proof of the theorem.
\end{proof}
 
 To illustrate the applicability of Theorem \ref{characterization} from a computational point of view, we next present a series of explicit numerical examples elaborating the different features of the best coapproximation problem, related to the existence and the uniqueness. In each case, an algorithmic approach is presented which further underlines the usefulness of the $ * $-Property in studying best coapproximation problems in subspaces of $ \mathcal{D}_n. $
 
 \begin{example}
 	Let $ A_1 = (( ~7, ~-5,~ 2,~ 6,~ -7,~ -5,~ 1~)),~ A_2 = ((~ 1,~ 3, ~4, ~3, ~-1,~ 3,$ $ ~ 2~)),~ A_3 = ((~ 3, ~-7, ~-4, ~5, ~-3, ~-7, ~-2~)) $ be linearly independent matrices in $ \mathcal{D}_7.$ Our aim is to find the best coapproximation(s) to any given $T$ out of $ \mathbb{{Y}} = span\left\lbrace A_1,A_2,A_3 \right\rbrace. $ In view of the Theorem \ref{characterization}, we proceed in the following steps.\\
 	
 	$ Step ~1:$ For $ i \in \left\lbrace 1, 2,\ldots, 7 \right\rbrace ,$ the $i$-th components are  respectively 
 	\[ (7, 1, 3), (-5, 3, -7), (2, 4, -4),(6, 3, 5), (-7, -1, -3), (-5,3,-7), (1,2,-2).\]
 	
 	$ Step ~2:$ $ P_1^+ = \left\lbrace 1\right\rbrace , P_1^- = \left\lbrace 5\right\rbrace ; ~ P_2^+ = \left\lbrace 2, 6\right\rbrace , P_2^- = \phi ; ~ P_3^+ = \left\lbrace 3\right\rbrace , P_3^- = \phi; ~ P_4^+ = \left\lbrace 4\right\rbrace , P_4^- = \phi;  ~P_5^+ = \left\lbrace 5\right\rbrace , P_5^- = \left\lbrace 1 \right\rbrace ; ~ P_6^+ = \left\lbrace 2, 6\right\rbrace , P_6^- = \phi; ~ P_7^+ = \left\lbrace 7\right\rbrace , P_7^- = \phi,  $  respectively, where $P_i^+$ and   $P_i^-$ are the positively associated set and the negatively associated set of the $i$-th component, respectively, for all $ i \in \left\lbrace 1, 2, \ldots ,7\right\rbrace.$\\
 	
 	$ Step ~3:$ Here the nonequivalent components satisfying the $*$-Property may be taken as the $1$-st component, the $2$-nd component, the $3$-rd component and the $4$-th component.\\
 	
 	$ Step ~4:$  In this final step, we consider a given $ T \in \mathcal{M}_7 $ and apply Theorem \ref{characterization} to obtain the best coapproximation to $ T $ out of $ \mathbb{{Y}}. $ In order to illustrate the various possibilities arising in the best coapproximation problem in $ \mathcal{D}_7, $ it suffices to consider the following three particular cases.\\
 	
 	$Case~1:$
 	Let $ T_1 \in \mathcal{M}_7 $ be given by  $T_1 = \left( b_{ij}\right)_{1 \leq i,j \leq 7}$  , where $b_{11} = 2, ~b_{15}=4,~ b_{22}=1,~ b_{26}= 3,~ b_{33}= 4,~ b_{44}= 1,~ b_{51}= -7,~ b_{55}= -2, ~b_{62}= 2, ~b_{66}= 1$ and the other $b_{ij} $'s can be chosen arbitrarily.\\

 	Therefore, $^1T_{1*} =\begin{pmatrix}
 	2 & 4\\
 	7 & 2
 	\end{pmatrix},~~ ^2T_{1*} = \begin{pmatrix}
 	1 & 3\\
 	2 & 1
 	\end{pmatrix}, ~ ^3T_{1*} = (4), ~ ^4T_{1*}= (1).$ \\
 	Then from Theorem \ref{characterization}, $\sum_{i=1}^{3}\alpha_iA_i $ is a best coapproximation to $T_1$ out of $ \mathbb{{Y}}$ if and only if $ \alpha_1, \alpha_2,\alpha_3 \in \mathbb{R}$ satisfies the following relations:
 	\begin{eqnarray*}
 	 7\alpha_1 + ~\alpha_2 + ~3\alpha_3 \in~ W\left( ^1T_{1*}\right)&=&[-7/2,~ 15/2] 
 	\\-5\alpha_1 +~ 3\alpha_2 - ~7\alpha_3 \in~ W\left( ^2T_{1*}\right)&=&[-3/2, ~7/2]
 \\	2\alpha_1 +~ 4\alpha_2 - ~4\alpha_3 \in~ W\left( ^3T_{1*}\right)&=& ~ \{4\}
 \\	6\alpha_1 +~ 3\alpha_2 - ~5\alpha_3 \in~ W\left( ^4T_{1*}\right)&=& \{1\}.
 	 \end{eqnarray*}
  Since there are infinitely many  $ \alpha_1, \alpha_2,\alpha_3 \in \mathbb{R}$ satisfying the above relations, there are infinitely many best coapproximation to $T_1$ out of $ \mathbb{{Y}}.$ Moreover,
   \[\ \mathcal{R}_\mathbb{{Y}}(T_1) =  \{((~x, ~4-x, ~4,~ 1,~ -x,~ 4-x~,2))~: ~ 1/2 \leq x \leq 11/2\}.\]
 $Case~2:$ 	Let $ T_2 \in \mathcal{M}_7 $ be given by  $T_1 = \left( c_{ij}\right)_{1 \leq i,j \leq 7}$  , where $~c_{11} = 3,~ c_{15}= -5, ~c_{22}=1, ~c_{26}= 3, ~c_{33}= 4, ~c_{44}= 1, ~c_{51}= -5, ~c_{55}= -3,~ c_{62}= 2,~ c_{66}= 1$ and the other $c_{ij}  $'s can be chosen arbitrarily.\\
  
  Therefore, $^1T_{2*} =\begin{pmatrix}
  3 & -5\\
  5 & 3
  \end{pmatrix},~~ ^2T_{2*} = \begin{pmatrix}
  1 & 3\\
  2 & 1
  \end{pmatrix}, ~ ^3T_{2*} = (4), ~ ^4T_{2*}= (1).$ \\
  Then from Theorem ~\ref{characterization},
  \begin{eqnarray*}
  7\alpha_1 + ~\alpha_2 + ~3\alpha_3 \in~ W\left( ^1T_{2*}\right)&=&~ \{3\} 
  \\-5\alpha_1 +~ 3\alpha_2 - ~7\alpha_3 \in~ W\left( ^2T_{2*}\right)&=&[-3/2, ~7/2]
  \\	2\alpha_1 +~ 4\alpha_2  ~-4\alpha_3 \in~ W\left( ^3T_{2*}\right)&=& ~\{4\}
  \\	6\alpha_1 +~ 3\alpha_2 - ~5\alpha_3 \in~ W\left( ^4T_{2*}\right)&=& \{1\}.
  \end{eqnarray*}
  Since there exist unique  $\alpha_1 ,\alpha_2 ,\alpha_3 \in \mathbb{R} $ satisfying the above relations, the best coapproximation to $T$ out of $ \mathbb{{Y}}  $ is unique.  Moreover,
   \[ \mathcal{R}_\mathbb{{Y}}(T_2) = \{((~3,~1,~4,~1,~-3,~1~,2))\}.\]  
   
   $Case~3:$ 	Let $ T_3 \in \mathcal{M}_7 $ be given by  $T_3 = \left( d_{ij}\right)_{1 \leq i,j \leq 7}$  , where $d_{11} = 14,~ d_{15}=1,~ d_{22}=1,~ d_{26}= 3,~ d_{33}= 4,~ d_{44}= 1,~ d_{51}= 1, ~d_{55}= -14,~ d_{62}= 2,~ d_{66}= 1$ and the other $d_{ij} $'s can be chosen arbitrarily.\\
  
  Therefore, $^1T_{3*} =\begin{pmatrix}
  14 & 1\\
  -1 & 14
  \end{pmatrix},~~ ^2T_{3*} = \begin{pmatrix}
  1 & 3\\
  2 & 1
  \end{pmatrix}, ~ ^3T_{3*} = (4), ~ ^4T_{3*}= (1).$ \\
  Then from Theorem ~\ref{characterization},
  \begin{eqnarray*}
    7\alpha_1 + ~\alpha_2 + ~3\alpha_3 \in~ W\left( ^1T_{3*}\right)&=&~\{14\}
  \\-5\alpha_1 +~ 3\alpha_2 - ~7\alpha_3 \in~ W\left( ^2T_{3*}\right)&=&[-3/2, ~7/2]
  \\	2\alpha_1 +~ 4\alpha_2 - ~4\alpha_3 \in~ W\left( ^3T_{3*}\right)&=& ~\{4\}
  \\	6\alpha_1 +~ 3\alpha_2 - ~5\alpha_3 \in~ W\left( ^4T_{3*}\right)&= &\{1\}.
  \end{eqnarray*}
  Since there exists no such  $\alpha_1 ,\alpha_2 ,\alpha_3 \in \mathbb{R} $ satisfying the above relations, it follows that \[ \mathcal{R}_\mathbb{{Y}}(T_3) = \phi. \] 
  \end{example}

  Our next goal is to obtain a tractable characterization of the coproximinal subspaces of $ \mathcal{D}_n $ with respect to $ \mathcal{M}_n. $ The following lemma is crucial for that purpose, besides being interesting in its own right by providing a lower bound on the number of nonequivalent components satisfying the $ * $-Property. 
  
\begin{lemma}\label{lemmapgeqm}
	Let $ \mathcal{A}= \{ A_1 , A_2 ,\ldots, A_m \} $ be linearly independent in $  \mathcal{D}_n ,$  where  $A_k= ((a^k_{11}, a^k_{22}, \ldots, a^k_{nn}))  ,$ for each  $ 1 \leq k \leq m. $ 
 Let the total number of nonequivalent components  satisfying the $*$-Property be $p$. Then $ p \geq m. $ 
\end{lemma}

\begin{proof}
		 Suppose that the $j_1$-th, $j_2$-th, $\ldots ,~j_p$-th nonequivalent components satisfy  the $*$-Property. Suppose on the contrary that $p < m.$  Let $ Y_1= span \{ ( a_{j_sj_s} ^1 ,a_{j_sj_s}^2, $ $\ldots,  a_{j_sj_s} ^m ): 1 \leq s \leq p \} $ and let $Y_2=  span \{\left( a_{ii} ^1 ,a_{ii}^2 ,\ldots,a_{ii} ^m \right)) :1 \leq  i \leq n\}.$ Clearly, $ Y_1 \subsetneq Y_2 = \mathbb{R}^m ,$  which implies that $ Y_2^\perp \subsetneq Y_1^\perp .$  Therefore, there exists $  \left( \gamma_1,\gamma_2,\ldots ,\gamma_m \right) \in Y_1^\perp \setminus Y_2^\perp$ such that $ | \sum_{k=1}^{m} \gamma_ka_{ii}^k | > max     \{ | {\sum_{k=1}^{m} \gamma_k a_{j_sj_s}^k} |:1 \leq s \leq $ $ p\} = 0 ,$ for some $i \notin \{ j_1, j_2, \ldots,$ $ j_p \} . $ Following Theorem \ref{existence}, we obtain that the $i$-th component, which is nonequivalent to the  $j_1$-th, $j_2$-th, $\ldots ,~j_p$-th components, satisfies the $*$-Property. This contradiction completes the proof of the lemma. 

\end{proof}

We now obtain a characterization of the coproximinal subspaces of $ \mathcal{D}_n $ in terms of the $ * $-Property.

\begin{theorem}\label{*property}
	Let $ \mathcal{A}= \{ A_1 , A_2 ,\ldots, A_m \} $ be linearly independent in $  \mathcal{D}_n ,$  where  $A_k= ((a^k_{11}, a^k_{22}, \ldots, a^k_{nn}))  ,$ for each  $ 1 \leq k \leq m. $  Then  $span\left\lbrace A_1 , A_2 ,\ldots, A_m\right\rbrace $ is a coproximinal subspace of $\mathcal{M}_n$ if and only if there exist exactly  $m$ number of nonequivalent  components satisfying the $*$-Property.
\end{theorem}

\begin{proof}
	Let us first prove the sufficient part of the theorem. Let the $j_1$-th, $j_2$-th, $\ldots,  ~j_m$-th components be chosen as the nonequivalent $m $ number of components satisfying  the $*$-Property.  Let us consider $C \in \mathcal{M}_m $ given by $C=(c_{st})_{1\leq s,t \leq m}$ such that $c_{st} = a_{j_sj_s}^t,$ where $\left( a_{j_sj_s}^1, a_{j_sj_s}^2, \ldots,a_{j_sj_s}^m\right) $ is the $j_s$-th component.  We claim that $rank(C)=m.$ Suppose on the contrary $rank(C)<m$. Let $Y_1= span \{\left( a_{j_sj_s} ^1 ,a_{j_sj_s}^2 ,\ldots,a_{j_sj_s} ^m \right): 1 \leq s \leq m \}  $ and let $Y_2=  span \{\left( a_{ii} ^1 ,a_{ii}^2 ,\ldots,a_{ii} ^m \right) :1 \leq  i \leq n\}.$ Clearly, $ Y_1 \subsetneq Y_2 = \mathbb{R}^m ,$ which implies that $ Y_2^\perp \subsetneq Y_1^\perp .$  Therefore, there exists $  \left( \gamma_1,\gamma_2,\ldots ,\gamma_m \right) \in Y_1^\perp \setminus Y_2^\perp$ such that $ | \sum_{k=1}^{m} \gamma_ka_{ii}^k | > max     \{ | {\sum_{k=1}^{m} \gamma_k a_{j_sj_s}^k} |:1 \leq $ $ s \leq m  \} = 0 ,$ for some $i \notin \left\lbrace j_1, j_2, \ldots, j_m\right\rbrace . $ Following Theorem \ref{existence}, there exists an  $i$-th component, which is nonequivalent to the  $j_1$-th, $j_2$-th, $\ldots , j_m$-th components, that satisfies the $*$-Property. This contradiction establishes our claim.  Therefore, $C$ is invertible and hence onto. So, for any $\beta_1, \beta_2, \ldots, \beta_m \in \mathbb{R}, $ there always exist $\alpha_1, \alpha_2, \ldots, \alpha_m \in \mathbb{R}$ such that $	a_{{j_s}{j_s}}^1 \alpha_1 +  a_{{j_s}{j_s}}^2 \alpha_2 +\ldots + a_{{j_s}{j_s}}^m \alpha_m = \beta_i, $ for all $s \in \{1, 2, \ldots, m\}$. Noting that for any $T \in \mathcal{M}_n,~ W(^{j_s}T_*) \subset \mathbb{R}$, therefore we conclude that 
	  \begin{align}
	  a_{{j_s}{j_s}}^1 \alpha_1 +  a_{{j_s}{j_s}}^2 \alpha_2 +\ldots + a_{{j_s}{j_s}}^m \alpha_m \in W \left( ^{j_s}T_*  \right) ,
	  \end{align}
	  for all $ s \in \left\lbrace  1, 2, \ldots, m\right\rbrace .$ Following Theorem \ref{characterization}, it is now evident that $\sum_{k=1}^{m}\alpha_kA_k $ is the best coapproximation to $T$ out of  $span\left\lbrace A_1 , A_2 ,\ldots, A_m\right\rbrace $. 
	  This shows that $span\left\lbrace A_1 , A_2 ,\ldots, A_m\right\rbrace $ is a coproximinal subspace of $\mathcal{M}_n.$  \\
	
	Let us now prove the necessary part of the theorem. Suppose that the $j_1$-th, $j_2$-th, $\ldots,~ j_p$-th nonequivalent components satisfy  the $*$-Property. Then from Lemma \ref {lemmapgeqm}, we get that $p \geq m.$ Let us now take the $p \times m $ matrix $D=(d_{st})_{1\leq s \leq p,~ 1\leq t \leq m}$ such that $d_{st} = a_{j_sj_s}^t,$ where $ \left( a_{j_sj_s}^1,  a_{j_sj_s}^2, \ldots,  a_{j_sj_s}^m \right) $ is the $j_s$-th component. Let $T_D \in \mathbb{L}(\mathbb{H}_1,\mathbb{H}_2)$ be the linear operator corresponding to the matrix $D$ with respect to the standard ordered bases of $ \mathbb{H}_1, \mathbb{H}_2, $ where  $\mathbb{H}_1= \mathbb{R}^m $ and $ \mathbb{H}_2 = \mathbb{R}^p$.  Since  $span\left\lbrace A_1 , A_2 ,\ldots, A_m\right\rbrace $ is a coproximinal subspace of $\mathcal{M}_n,$ for  $T \in \mathcal{M}_n, $ there exist $ \alpha_1, \alpha_2,\ldots, \alpha_m \in \mathbb{R}$ satisfying the following relations:
		\begin{align}
	a_{{j_s}{j_s}}^1 \alpha_1 +  a_{{j_s}{j_s}}^2 \alpha_2 +\ldots + a_{{j_s}{j_s}}^m \alpha_m \in W \left( ^{j_s}T_*  \right) ,
	\end{align}  
		for all $ s \in \left\lbrace  1, 2, \ldots, p\right\rbrace .$ We now claim that $T_D$ is onto.  Let $ \beta = (\beta_1, \beta_2, \ldots, \beta_p) \in \mathbb{R}^p.$  We note that for any $T=(b_{ij})_{ 1 \leq i,j \leq n} ,~ ^{j_s}T_*$ is a $h \times h $ matrix whose entries are $( \pm b_{ij})$ depending on $ P_{j_s}^+$ and $P_{j_s}^-,$ where $ |P_{j_s}^+ \cup P_{j_s}^-|= h$.  So we can choose $T$ suitably so that  $W(^{j_s}T_*)= \{ \beta_s\}$  for each $ s \in \left\lbrace  1, 2, \ldots, p\right\rbrace.$ This shows that for each  $ s \in \left\lbrace  1, 2, \ldots, p\right\rbrace,$  we get 
		\[ a_{{j_s}{j_s}}^1 \alpha_1 +  a_{{j_s}{j_s}}^2 \alpha_2 +\ldots + a_{{j_s}{j_s}}^m \alpha_m = \beta_s\]
		and so $T_D (\alpha) = \beta, $ where $ \alpha = ( \alpha_1, \alpha_2, \ldots \alpha_m) \in \mathbb{R}^m.$ 
		Thus  $T_D$ is onto and therefore, $m \geq p.$  This along with  Lemma  $\ref{lemmapgeqm}$ completes the proof.
\end{proof}

We now obtain a characterization of the co-Chebyshev subspaces of $ \mathcal{D}_n $ in terms of the $ * $-Property.

\begin{theorem}\label{cochebyshev}
	Let $ \mathcal{A}= \{ A_1 , A_2 ,\ldots, A_m \} $ be linearly independent in $  \mathcal{D}_n ,$  where  $A_k= ((a^k_{11}, a^k_{22}, \ldots, a^k_{nn}))  ,$ for each  $ 1 \leq k \leq m. $  Suppose that the $ i_1$-th, $i_2$-th, $\ldots$, $i_p$-th nonequivalent components  satisfy the $*$-Property.
		 Then $ span\{A_1 , A_2 ,\ldots, A_m \}$ is a co-Chebyshev subspace of $\mathcal{M}_n$ if and only if $p = m$ and $ |P_{i_s}^+\cup P_{i_s}^-|=1$ for all $ s \in \{ 1, 2, \ldots, p\}. $
\end{theorem}

\begin{proof}
We first prove the sufficient part of the theorem. Since $ p = m$, we note from Theorem  \ref{*property} that $ span\{A_1 , A_2 ,\ldots, A_m \}$ is a coproximinal subspace of $\mathcal{M}_n$. Therefore, for any given $T \in \mathcal{M}_n$, there exist $ \alpha_1, \alpha_2. \ldots, \alpha_m \in \mathbb{R}$ satisfying the following relations:
	\begin{align}\label{relation}
	a_{i_si_s}^1 \alpha_1 +  a_{i_si_s}^2 \alpha_2 +\ldots + a_{i_si_s}^m \alpha_m \in W \left( ^{i_s}T_*  \right) ,
	\end{align} 
	for all $ s \in \{ 1, 2, \ldots, m\}$. Since $ |P_{i_s}^+\cup P_{i_s}^-|=1$, it follows that $ ^{i_s}T_*$ is of order $1$. Moreover, $ W \left( ^{i_s}T_*  \right) = (b_{i_si_s}),$ for every $s \in \{ 1, 2, \ldots, m\}, $ where $ T = (b_{ij})_{1 \leq i,j \leq n}$. Therefore,  relations (\ref{relation}) represent a system of linear equation with coefficient matrix $C = (c_{st})_{1 \leq s,t \leq m},$ where $ c_{st} = a_{i_si_s}^t .$ Following the arguments given in the proof  of Theorem \ref{*property}, we conclude that $C$ is invertible. Hence for any given $T \in \mathcal{M}_n,$ there exists a unique $(\alpha_1, \alpha_2, \ldots, \alpha_m)\in \mathbb{R}^m $ satisfying the relations (\ref{relation}).  Therefore, $ span\{A_1 , A_2 ,\ldots, A_m \}$ is a co-Chebyshev subspace of $\mathcal{M}_n. $\\ 
	
	 We now prove the necessary part of the theorem. Let us assume that $span\{A_1, $ $ A_2 ,\ldots, A_m \}$ is a co-Chebyshev subspace of $\mathcal{M}_n. $ In particular, $span\{A_1 , A_2 ,\ldots, A_m \}$ is a coproximinal subspace of $\mathcal{M}_n. $ Therefore, from Theorem \ref{*property}, we get $p = m.$ Suppose on the contrary $ |P_{i_s}^+\cup P_{i_s}^-|=k_s>1,$ for some $ s \in \{1, 2, \ldots, m\}.$ Therefore, $ ^{i_s}T_*$ is of order $k_s$. Let us consider  $ Q =(q_{rt})_{1 \leq r \leq p, 1\leq t \leq m},$ where $q_{rt}= a_{i_ri_r}^t.$ Since $p=m$, following the arguments given in the proof of Theorem \ref{*property}, $ Q $ is invertible. Now, for any two scalars $c$ and $ d$, where $c \neq d,$ we can choose a suitable $T \in \mathcal{M}_n$ such that $c, d \in  W \left( ^{i_s}T_*  \right).$ Therefore, we can conclude that there exist at least two different sets of  $(\alpha_1,\alpha_2, \ldots , \alpha_m) \in \mathbb{R}^m $ satisfying the relations:
	  	\[a_{i_li_l}^1 \alpha_1 +  a_{i_li_l}^2 \alpha_2 +\ldots + a_{i_li_l}^m \alpha_m \in W \left( ^{i_l}T_*  \right),\]
	for all $ l \in \{ 1, 2, \ldots, m\}$. This contradicts that $span\{A_1 , A_2 ,\ldots, A_m \}$ is a co-Chebyshev subspace of $\mathcal{M}_n. $ Hence the theorem.

\end{proof}

As an immediate application of the above theorem, we record the following interesting observation.

\begin{cor}\label{e}
	$\mathcal{D}_n$ is a co-Chebyshev subspace of $\mathcal{M}_n.$	
\end{cor}

\begin{proof}
	Clearly, $\{ E_k : 1\leq k \leq n \}$ is a basis of $\mathcal{D}_n,$ where  $E_k = (e_{ij}^k)_{1 \leq i,j \leq n} $ is given by 
	
	\begin{eqnarray*}
		e_{ij}^k  & = &  1,~ \mbox{whenever} ~ i = j= k \\ 
		& = &0 ,~ \mbox{otherwise}.
	\end{eqnarray*}
	It is trivial to observe that for all $ 1 \leq i \leq n, $ the $ i $-th component satisfies the $ * $-Property and $|P_i^+ \cup P_i^-|=1$. Therefore, the desired result follows directly from Theorem \ref{cochebyshev}.
\end{proof}

As another important application of the theories developed in the present article, it is possible to characterize the best coapproximation problem in the setting of $ \ell_{\infty}^n, $ for any given $ n \in \mathbb{N}. $ This in turn is equivalent to the following optimization problem:\\

\textbf{Problem:} Let $ a_{ij}, \alpha_k \in \mathbb{R} $ be fixed, where $ 1\leq i \leq m, 1 \leq j, k \leq n. $ Find a necessary and sufficient condition for the existence of $ \beta_1, \beta_2, \ldots, \beta_m \in \mathbb{R} $ such that for any $ c_1, c_2,\ldots, c_m \in \mathbb{R},$  the following inequality holds true:
\begin{equation*}
max\left\lbrace  |\alpha_k - \sum_{i=1}^{m} c_ia_{ik}| : 1 \leq k \leq n \right\rbrace  \geq max \left\lbrace |\sum_{i=1}^{m}\beta_ia_{ik} - \sum_{i=1}^{m} c_ia_{ik}| : 1 \leq k \leq n\right\rbrace .
\end{equation*} 

Moreover, in case existence is guaranteed, find all such $ \beta_1, \beta_2, \ldots, \beta_m \in \mathbb{R}. $ \\

 We  emphasize that the above problem is not entirely trivial, most notably because the existence of a desired solution is not a priori guaranteed. However, it is possible to completely solve the problem (from both theoretical and computational perspectives), by applying the methodology already developed in this article. It is well-known that $ \ell_{\infty}^n $ (endowed with its usual maximum norm) is isometrically isomorphic to $ \mathcal{D}_n $ endowed with the usual operator norm. Indeed, the natural choice map $ \Psi : \ell_{\infty}^n \longrightarrow \mathcal{D}_n, $ taking $ (a_1, a_2, \ldots, a_n) \in \ell_{\infty}^n $ to $ ((a_1, a_2, \ldots, a_n)) \in \mathcal{D}_n $ is easily seen to be the desired isometric isomorphism. This connection allows us to obtain an algorithmic approach to the best coapproximation problem in any given subspace $ \mathbb{Y} $ of $ \ell_{\infty}^n $ via the methods already developed to treat the corresponding best coapproximation problem in the subspace $ \Psi(\mathbb{Y}) $ of $ \mathcal{D}_n. $ It should be noted in this context that our theory essentially translates into characterizing the best coapproximation(s) to any given $ T \in \mathcal{M}_n $ out of any given subspace of $ \ell_{\infty}^n, $ and therefore, the best coapproximation problem in subspaces of $ \ell_{\infty}^n $ is only a particular case of the results developed so far. To illustrate this further, we make note of the following two remarks pertaining to the best coapproximation problem in $ \ell_{\infty}^n: $ \\

\begin{itemize}
	\item $ \ell_{\infty}^n $ is a coproximinal subspace of $ \mathcal{M}_n $ for each $ n \in \mathbb{N}. $ This is simply a reformulation of Corollary \ref{e}.
	\\
	
	\item By using the concept of the $ * $-Property, and the above mentioned isometric isomorphism $ \Psi : \ell_{\infty}^n \longrightarrow \mathcal{D}_n, $ it is quite straightforward to construct subspaces of $ \ell_{\infty}^n $ which are (not) coproximinal. Indeed, in light of the theories developed in this article, any such construction essentially reduces to controlling the number (nonequivalent) of $ i $-th components satisfying the  $ * $-Property, for $ 1 \leq i \leq n. $ As an explicit example, it can be readily verified that $ \mathbb{Y}_1 $ is a coproximinal subspace of $ l_\infty^7, $ whereas $ \mathbb{Y}_2 $ is not, where
	
	\[ \mathbb{Y}_1 = span\left\lbrace ( 6, 1, 4, 3, 3, 1, 1), ( 2, 5, 2, 3, 1, 5, 1), (4, 3, 8, 6, 2, 3, 2), ( 2, 1, 4, 9, 1, 1, 3)  \right\rbrace,  \]
	\[ \mathbb{Y}_2= span\left\lbrace (2, -5, 3, 1, -2, -5, 2),  (-4, 2, 2, -2, -4, 2, -4)\right\rbrace. \] \\

\end{itemize}

In view of our treatment of the theory of coapproximations in $ \ell_{\infty}^n $ spaces, it seems appropriate to end the present article with the following concluding remark:

\begin{remark}
	The theory of best coapproximations in Banach spaces remains a much less exposed area of research, especially from a computational point of view, in comparison to the theory of best approximations. Certainly, this is in part due to the inherently complicated non-linear nature of the best coapproximation problem and the difficulty of the corresponding computations involved in the process. In this context, the reader is encouraged to look up the literature, including \cite{PS, RS}.  Our main focus in this article is to illustrate the following principle in the setting of $ \ell_{\infty}^n $ (or, more generally, for matrices in $ \mathcal{M}_n $ out of subspaces of $ \mathcal{D}_n $): \\
	
	\textit{It is possible to essentially reduce the much harder ``best coapproximation problem" to the well-known and way more simpler  ``existence and uniqueness problem corresponding to a particular system of linear problems", by applying the concept of orthogonality.}\\

	Indeed, using the methodology developed so far, it is now very easy to explicitly produce examples of coproximinal and co-Chebyshev subspaces in the setting of $ \ell_{\infty}^n. $ Therefore, in light of the above fact, a natural query would be to test the validity of such a nicety, in the setting of classical Banach spaces other than $ \ell_{\infty}^n. $

\end{remark}

\end{document}